\documentclass[12pt]{article}
\usepackage[margin=0.75in]{geometry}
\usepackage{microtype}
\usepackage{amsmath,amssymb,amsthm}
\usepackage{booktabs,graphicx,array}
\usepackage[numbers,sort&compress]{natbib}
\usepackage{xcolor}
\usepackage[colorlinks=true,linkcolor=blue!50!black,citecolor=blue!50!black,urlcolor=blue!50!black]{hyperref}
\usepackage{xurl}
\hypersetup{pdftitle={When Is Local Time Stepping Worthwhile? A Cost Analysis for Nested Adaptive Mesh Refinement},pdfauthor={James M. Hyman}}

\newtheorem{proposition}{Proposition}
\newtheorem{lemma}{Lemma}
\newtheorem{corollary}{Corollary}
\theoremstyle{remark}
\newtheorem{remark}{Remark}

\newcommand{\WG}{W_{\mathrm{glob}}}
\newcommand{\WL}{W_{\mathrm{sub}}}
\newcommand{\Sid}{S_{\mathrm{ideal}}}
\newcommand{\doi}[1]{doi: \href{https://doi.org/#1}{\nolinkurl{#1}}}

\title{When Is Local Time Stepping Worthwhile?\\[4pt]
\large A Cost Analysis for Nested Adaptive Mesh Refinement}
\author{James M. Hyman\\[2pt]
\small Department of Mathematics, Tulane University, New Orleans, LA 70118, USA\\
\small\texttt{mhyman@tulane.edu}}
\date{October 10, 2026}

\begin{document}
\maketitle

\begin{abstract}
In deep hierarchies of adaptive mesh refinement, the update work that subcycling avoids is set mainly by the dimension of what is refined, not of the space. Subcycling (local time stepping) does less work than global stepping, which advances every level with the finest step, exactly when its extra coordination cost stays below an overhead budget $\beta=1-1/\Sid$ as a fraction of the work of global stepping. Here $\Sid$ is the ideal speedup that the cell counts give for equal update costs. When each level follows a feature of box-counting dimension $\delta$ and the counts grow geometrically, the budget of a code that advances every cell tends to $(r-1)/(r^{\delta+1}-1)$ for refinement ratio $r$. A point leaves a budget near one, while a surface in three dimensions leaves $1/7$ at $r=2$ and $1/21$ at $r=4$, values implied by the planar estimates of Dursi and Zingale. Published comparisons in discontinuous Galerkin and octree codes with local time stepping imply overheads of 0.1 to 5 percent of the work, below $1/7$ and comparable with $1/21$. In paired runs of one AMReX code on one core, a cost model fixed in advance and fitted on points and filled regions predicted the speedups of curves and surfaces with a median error of 7 percent (22 percent on the saving); post hoc, the update counts alone did as well except for points. The overhead as defined here was negative in all 17 configurations, but computed from speedups like the published values it ranged from $-0.11$ to $0.19$, and no configuration came near break-even. When an update costs the same under both schedules, one profiled run can show that subcycling pays.
\end{abstract}


\section{Introduction}\label{sec:intro}

Practice is divided. Applications built on the \texttt{Amr} class of AMReX, a framework for block-structured adaptive mesh refinement (AMR)~\cite{zhang2019amrex}, subcycle by default and can switch subcycling off with a runtime parameter~\cite{amrex2026runtime}, and the Castro code offers the same default along with a manual schedule~\cite{castro2026timestepping}. Nyx, also built on AMReX, chooses the most efficient subcycling pattern across levels at set intervals during a run~\cite{almgren2013nyx,amrex2026source}. A survey of block-structured frameworks found that all of them except FLASH support subcycling~\cite{dubey2014survey}. Athena++ instead advances all levels with a single time step~\cite{stone2020athena}. For hierarchies with equal numbers of cells on every level, its authors estimate that subcycling would reduce the work by a factor of only about half the number of levels. They also note that interpolation in time adds error and that the savings may be offset by reduced efficiency on parallel machines. Unlike its subcycled counterpart PeleLM, the low Mach number code PeleLMeX also advances all levels together, a choice that let its developers extend a closed-chamber algorithm, in which the thermodynamic pressure of a closed domain evolves in time, to AMR~\cite{esclapez2023pelelmex}. These defaults are set once and paid for in every run that uses them, usually without a measured value for the cost of coordinating the levels.

When does the work saved on the coarse levels pay for the cost of coordinating them? In the nested-grid method of Berger and Oliger~\cite{berger1984adaptive}, a level refined by the ratio $r$ in space also takes $r$ time steps for each step of the level above it. Each level then advances with a step set by its own stability limit. Berger and Colella~\cite{berger1989local} made the coupling between levels conservative by correcting the coarse fluxes at coarse--fine boundaries with the accumulated fine fluxes. Refinement in time is called subcycling in the AMR literature, local time stepping for partial differential equations more generally~\cite{osher1983numerical,dumbser2014arbitrary}, and multirate time stepping for ordinary differential equations~\cite{savcenco2007multirate}. The alternative, which we call global stepping, is also called global time stepping or non-subcycling, and Li and Hyman called it a locked time step~\cite{li2003adaptive}. It advances every level with the time step of the finest level.

Li and Hyman met the question while building a fifth-order WENO scheme into AMR. They locked the time step so that coarse data never had to be interpolated in time. Citing an earlier report of theirs, they found the locked step only slightly slower than local time stepping~\cite{li2003adaptive}. Dursi and Zingale~\cite{dursi2005efficiency} gave an early quantitative answer by computing the ratio of the work with and without time refinement for patch and tree meshes refined around a point, a planar interface and a filled circular region in two dimensions, with no overhead charged to subcycling. Their estimates show large savings only when a few points in a large domain need refinement, and modest savings otherwise. For a planar interface in three dimensions on a patch mesh with $r=2$, time refinement does about $6/7$ of the work. The same ratio of update counts measures the potential gain in studies of local time stepping~\cite{dumbser2014arbitrary,domingues2008adaptive,seny2014efficient}. Domingues et al.\ wrote the saving as a sum over levels weighted by the fraction of cells on each level~\cite{domingues2008adaptive}, and Fernando and Sundar compared the same multilevel model with measured speedups on octree grids~\cite{fernando2022scalable}. The cell counts themselves follow a box-counting law, which Kritsuk, Norman and Padoan used to relate the cells on each AMR level to the fractal dimension of shocks in supersonic turbulence~\cite{kritsuk2006adaptive}. Under that law the number of cells of width $\varepsilon$ that meet a feature grows like $\varepsilon^{-\delta}$ as $\varepsilon$ shrinks, where $\delta$ is the box-counting dimension of the feature. The same power law governs the blocks of a quadtree around an object with a fractal boundary~\cite{faloutsos1996analysis}.

This note combines these strands, and it adds four things that the earlier analyses do not contain. (i) Section~\ref{sec:model} charges subcycling an explicit overhead $H$ and writes the break-even condition, for two levels and for any prescribed schedule, as an overhead budget $\beta$. (ii) Section~\ref{sec:dimension} joins the box-counting law to the multilevel identity to give closed forms for $\beta$ for features of any dimension $\delta$. They cover codes that advance every cell of a level and codes that advance only uncovered cells, and a corollary bounds $\beta$ when the counts are only near geometric. The closed forms extend the estimates of Dursi and Zingale for points and planar interfaces to curves in three dimensions, curved surfaces and fractal fronts. (iii) Section~\ref{sec:published} places the overheads implied by published measurements against these budgets. (iv) Section~\ref{sec:amrex} tests the accounting in paired runs of one AMReX code under an analysis plan fixed before any timing. It compares the measured speedups with the update counts and with a fitted cost model, and it examines a test that needs only one subcycled run; Appendix~\ref{app:protocol} records the protocol. Section~\ref{sec:limits} states what the work model leaves out: time on parallel machines, elliptic synchronization and accuracy. Throughout, the contribution is an accounting framework rather than a new integration method.

The analysis also sharpens a natural intuition, that subcycling pays in one dimension but not in three. The fraction of fine cells does grow with dimension when the refined volume is held fixed. In a deep hierarchy what matters instead is mainly the dimension of the refined feature, and a point in three dimensions offers at least the savings of a point in one.

\begin{center}
\fbox{\parbox{0.94\linewidth}{\textbf{Decision rule.} Let $\Sid$ be the ratio of the cell updates of global stepping to those of subcycling, and let $\beta=1-1/\Sid$. When an update costs the same on every level and under both schedules, subcycling does less work exactly when its extra cost of coordinating the levels, as a fraction of the work of global stepping, stays below $\beta$ (Proposition~\ref{prop:break} and Section~\ref{sec:multi}). For a code that advances every cell of a hierarchy following a feature of box-counting dimension $\delta$, with cell counts that grow geometrically and time refined like space, $\beta$ tends in a deep hierarchy to $(r-1)/(r^{\delta+1}-1)$ (Proposition~\ref{prop:geom}):\par\smallskip
\centering
\begin{tabular}{lcccc}
\toprule
 & point, $\delta=0$ & curve, $\delta=1$ & surface, $\delta=2$ & filled region, $\delta=3$\\
\midrule
$r=2$ & $\to1$ & $1/3$ & $1/7$ & $1/15$\\
$r=4$ & $\to1$ & $1/5$ & $1/21$ & $1/85$\\
\bottomrule
\end{tabular}\par\smallskip
\raggedright Overheads of a few percent, as published comparisons in other local time stepping codes suggest (Section~\ref{sec:published}), consume most of the budget of a surface at $r=4$.}}
\end{center}

\section{Work model and overhead budget}\label{sec:model}

\subsection{Two levels}\label{sec:two}

Consider a two-level mesh with refinement ratio $r=4$ in which nine of every ten cells are fine. Global stepping advances every cell four times during one coarse step, while subcycling advances the coarse cells only once. Subcycling therefore avoids three updates of each coarse cell, or $0.3$ of a sweep over all cells, out of the four sweeps that global stepping performs. The saving is $7.5\%$ of the work, and an extra cost of subcycling any larger than that erases it.

To state this in general, let $N_c>0$ and $N_f>0$ be the numbers of cells advanced on the coarse and fine levels during each of their steps. Patch-based codes usually advance the coarse cells that lie under fine patches as well and later replace their values by averages of the fine data. These cells then belong in $N_c$, while codes that advance only uncovered cells leave them out. During one coarse interval $\Delta T$, subcycling advances the fine level in $r$ steps of length $\Delta T/r$, while global stepping takes $r$ such steps on both levels. Let $C>0$ be the cost of one cell update in the kernel that advances a level. Ghost-cell filling, interpolation, flux bookkeeping and synchronization are not part of $C$. Their net extra cost under subcycling over one coarse interval is the overhead $H$, in which costs common to both schedules cancel and which may be negative. The work per coarse interval is then
\begin{equation}\label{eq:work}
\WG = rC(N_c+N_f), \qquad \WL = C(N_c+rN_f)+H .
\end{equation}

\begin{proposition}[Break-even condition]\label{prop:break}
Subcycling does less work than global stepping, $\WL<\WG$, if and only if
\begin{equation}\label{eq:break}
H<(r-1)CN_c .
\end{equation}
Equality gives equal work.
\end{proposition}

\begin{proof}
Subtracting the two expressions in~\eqref{eq:work} gives $\WG-\WL=(r-1)CN_c-H$.
\end{proof}

With the fine fraction $f=N_f/(N_c+N_f)$ and the normalized overhead $h=H/[C(N_c+N_f)]$, which measures $H$ in units of one sweep over all cells, the ratio of the two costs is
\begin{equation}\label{eq:speed}
S=\frac{\WG}{\WL}=\frac{r}{1+(r-1)f+h},\qquad \Sid=\frac{r}{1+(r-1)f}.
\end{equation}
The ideal speedup $\Sid$, the value at $H=0$, falls from $r$ toward one as $f$ grows, so a mesh dominated by fine cells leaves little to save. A second form of the criterion is more useful in practice, because it compares the overhead directly with the work that subcycling avoids. As a fraction of the work of global stepping, the avoided update work is
\begin{equation}\label{eq:budget}
\beta=1-\frac{1}{\Sid}=\Bigl(1-\frac{1}{r}\Bigr)(1-f),
\end{equation}
and Proposition~\ref{prop:break} is equivalent to $H/\WG<\beta$. The quantity $\beta$ is an overhead budget: subcycling wins exactly when its extra cost, measured as a fraction of the work of global stepping, stays below $\beta$. When $H<0$ the realized saving $\beta-H/\WG$ exceeds $\beta$, so $\beta$ is the saving before overhead rather than a ceiling. The same quantity, read as a speed-up rate, appears in the analysis of multiresolution schemes with local time stepping~\cite{domingues2008adaptive}. The reading as a budget against a measured overhead is the point of the present note.

In the example above, $r=4$ and $f=0.9$ give $\Sid=1.081$ and $\beta=0.075$. An overhead of half a sweep per coarse step ($h=0.5$) amounts to $H/\WG=h/r=0.125$. That exceeds the budget and gives $S=0.952$, so global stepping is cheaper even though it performs more cell updates. The same overhead with $f=0.1$ leaves $S=2.22$, because there the budget is $\beta=0.675$.

\begin{remark}[Normalization]\label{rem:norm}
The verdict does not depend on which run the overhead is measured against. With $x=H/\WG$ one has $H/\WL=x/(1-\beta+x)$, and $x<\beta$ holds exactly when $H/\WL<\beta$. A profiler that reports the overhead as a fraction of the subcycled run therefore gives the same answer at break-even.
\end{remark}

\begin{remark}[Level-dependent costs]\label{rem:costs}
If coarse and fine updates cost $C_c$ and $C_f$, the work becomes $\WG=r(C_cN_c+C_fN_f)$ and $\WL=C_cN_c+rC_fN_f+H$, and the criterion becomes $H<(r-1)C_cN_c$. Only the cost of the coarse updates that subcycling avoids enters the threshold.
\end{remark}

\begin{remark}[Schedule-dependent costs and counts]\label{rem:schedule}
The model assumes that an update costs the same under both schedules. When stiff or implicit terms are solved iteratively, the cost of a step can depend on its length. Each of the $r$ short coarse steps of global stepping may then cost less than one long step. With coarse costs $C_c^{\mathrm{G}}$ and $C_c^{\mathrm{L}}$ under the two schedules, the criterion becomes $H<(rC_c^{\mathrm{G}}-C_c^{\mathrm{L}})N_c$. The counts can also differ because under global stepping no fine step needs coarse data inside the fine region. A patch code could then skip covered coarse cells under global stepping while still advancing them under subcycling. The counts of each schedule then belong in its own expression, and the budget shrinks.
\end{remark}

\subsection{Multilevel schedules}\label{sec:multi}

Now let the hierarchy have levels $\ell=0,\dots,L$, with $N_\ell$ cells advanced on level $\ell$ at a cost $C_\ell$ per update. During one step of level~0, a subcycled schedule advances level $\ell$ a number $q_\ell$ of times, with $q_0=1$. Global stepping advances every level $Q=\max_\ell q_\ell$ times, which equals $q_L=r^L$ for a uniform ratio $r$ with time refined like space. With $H$ again the net overhead of subcycling,
\begin{equation}\label{eq:multiwork}
\WG=Q\sum_{\ell=0}^{L}C_\ell N_\ell,\qquad \WL=\sum_{\ell=0}^{L}q_\ell C_\ell N_\ell+H,
\end{equation}
so that
\begin{equation}\label{eq:multibreak}
\WL<\WG \quad\Longleftrightarrow\quad H<\sum_{\ell=0}^{L}(Q-q_\ell)\,C_\ell N_\ell .
\end{equation}
Every level with $q_\ell<Q$ contributes to the right side of~\eqref{eq:multibreak}, weighted by the number of updates it avoids. The budget form carries over unchanged, with $\Sid=Q\sum_\ell C_\ell N_\ell/\sum_\ell q_\ell C_\ell N_\ell$ and $\beta=1-1/\Sid$. With uniform costs and $q_\ell=r^\ell$, this $\Sid$ is the reciprocal of the work ratio computed by Dursi and Zingale~\cite{dursi2005efficiency} and the speedup model of Fernando and Sundar~\cite{fernando2022scalable}. The identity adds level-dependent costs, an arbitrary schedule and the explicit overhead to that model.

The identity does not assume that the ratios in time equal those in space. Stability can force a smaller ratio in time: the Athena++ authors note that the fastest waves can sit on the root level~\cite{stone2020athena}, and Nyx finds that not subcycling is most efficient when a global constraint limits every level's step~\cite{almgren2013nyx}. Suppose each level may take only $r_t<r$ steps per step of the level above, so that $q_\ell=r_t^\ell$ and $Q=r_t^L$. The closed forms of Section~\ref{sec:dimension} then hold with $r$ replaced by $r_t$ wherever it counts steps, while the growth factor per level $\rho=r^\delta$ of a feature of dimension $\delta$ (Section~\ref{sec:growth}) still counts cells, so that with every cell advanced the deep-hierarchy budget becomes $(r_t-1)/(r_t\rho-1)$. For a surface refined with $r=4$ in space but $r_t=2$ in time, the budget falls from $1/21$ (Table~\ref{tab:limits}) to $1/31$, and with $r_t=1$ there is nothing left to save.

\section{What sets the cell counts}\label{sec:dimension}

\subsection{A fixed refined volume}\label{sec:fixed}

Suppose first that a two-level mesh refines a fixed fraction $p\in(0,1)$ of the volume of a base grid with $N$ cells in $d$ dimensions, and ignore padding and boundary effects. Counting only uncovered coarse cells, $N_c=(1-p)N$ and $N_f=pr^dN$, so the fine fraction is
\begin{equation}\label{eq:fixed}
f_d(p,r)=\frac{pr^d}{1-p+pr^d}.
\end{equation}
If the covered coarse cells are also advanced, $N_c=N$ and $f_d=pr^d/(1+pr^d)$, which changes the numbers slightly but not the trend. For fixed $p$ and $r$ the fine fraction rises with $d$ and the ideal speedup falls. With $p=0.05$ and $r=4$, Table~\ref{tab:fixed} shows $\Sid$ dropping from $2.63$ in one dimension to $1.21$ in three. This is the dimension effect behind the intuition that subcycling suits one-dimensional problems. Yet a fixed refined volume is the wrong picture for most AMR calculations, because the refined region follows features of the solution and shrinks as the hierarchy deepens.

\begin{table}[ht]
\centering
\caption{Two levels refining a fixed fraction $p=0.05$ of the volume with $r=4$ and $H=0$, counting only uncovered coarse cells. The fine fraction $f_d$ comes from~\eqref{eq:fixed}, the ideal speedup from~\eqref{eq:speed} and the budget from~\eqref{eq:budget}. At a fixed refined volume the budget falls with the dimension of the space; Section~\ref{sec:feature} shows why this picture misleads. These are evaluations of the model, not timings.}
\label{tab:fixed}
\begin{tabular}{cccc}
\toprule
Dimension $d$ & Fine fraction $f_d$ & Ideal speedup $\Sid$ & Budget $\beta$\\
\midrule
1 & 0.1739 & 2.6286 & 0.6196\\
2 & 0.4571 & 1.6867 & 0.4071\\
3 & 0.7711 & 1.2073 & 0.1717\\
\bottomrule
\end{tabular}
\end{table}

\subsection{Growth law}\label{sec:growth}

Let level $\ell+1$ consist of the children of the level-$\ell$ cells whose centers lie within $b$ cell widths, $b\,\Delta x_\ell$, of a feature $F$, with $b$ fixed and $\Delta x_\ell=\Delta x_0\,r^{-\ell}$, and write $M_\ell$ for the number of cells that level $\ell$ holds. Suppose that the number $N_\varepsilon(F)$ of mesh cubes of side $\varepsilon$ that meet $F$ satisfies $c_1\varepsilon^{-\delta}\le N_\varepsilon(F)\le c_2\varepsilon^{-\delta}$ for $\varepsilon\le\varepsilon_0$, with constants $c_1,c_2>0$. Compact smooth curves and surfaces satisfy this bound, as do self-similar fractals that satisfy the open set condition, and $\delta$ is then the box-counting dimension of $F$~\cite{falconer2014fractal}.

\begin{lemma}[Growth law]\label{lem:growth}
If $b\ge\sqrt d/2$, every cell that meets $F$ is present on every level, and there are constants $k_1,k_2>0$, depending on $F$, $b$, $d$, $r$ and $\Delta x_0$, such that $k_1 r^{\delta\ell}\le M_\ell\le k_2 r^{\delta\ell}$ whenever $\Delta x_{\ell-1}\le\varepsilon_0$. In particular $M_\ell^{1/\ell}\to r^\delta$.
\end{lemma}

\begin{proof}
A cell of width $\Delta x$ that meets $F$ has its center within the half-diagonal $\sqrt d\,\Delta x/2\le b\,\Delta x$ of $F$, so it is tagged. By induction every level-$\ell$ cell that meets $F$ exists, because its parent meets $F$ and was tagged, so the tagged cells on level $\ell-1$ number at least $N_{\Delta x_{\ell-1}}(F)$. They lie within $(b+\sqrt d/2)\Delta x_{\ell-1}$ of $F$, so each lies within a fixed number of cell widths of a cell that meets $F$. They therefore number at most a constant times $N_{\Delta x_{\ell-1}}(F)$, the constant depending only on $b$ and $d$. The bounds follow because $M_\ell$ is $r^d$ times the tagged count on level $\ell-1$ and $\Delta x_{\ell-1}^{-\delta}=\Delta x_0^{-\delta}r^{\delta(\ell-1)}$.
\end{proof}

The lemma is the AMR form of the box-counting law that Kritsuk et al.\ applied to AMR levels~\cite{kritsuk2006adaptive} and that Faloutsos and Gaede derived for quadtrees under a level-independence assumption~\cite{faloutsos1996analysis}. Over several levels the counts therefore grow by the average factor
\begin{equation}\label{eq:growth}
\bar\rho=\Bigl(\frac{M_{\ell+k}}{M_\ell}\Bigr)^{1/k}\approx r^{\delta},
\end{equation}
where the approximation improves as $k$ grows; it is a scaling statement, not an exact count. In three dimensions a point, a smooth curve, a smooth surface and a filled region have $\delta=0$, $1$, $2$ and $3$. For a smooth feature the individual ratios approach $r^\delta$ once $b\,\Delta x_\ell$ is small compared with the radius of curvature, or for a filled region with its radius. For points and for features aligned with the grid they keep fluctuating, because the position of the feature within a cell changes from level to level. Padding fixed in cell units, such as buffer cells, blocking factors and the clustering of tagged cells into patches, changes $b$ and the constants but not the exponent. A front that is fractal over the resolved range of scales has a non-integer $\delta$ between 2 and 3; for the turbulent--non-turbulent interface of a supersonic boundary layer, box counting gives $\delta\approx2.31$~\cite{zhuang2018fractal}. Below the inner cutoff of the fractal range the front is smooth, so~\eqref{eq:growth} with non-integer $\delta$ holds only on the levels that resolve that range.

\subsection{Budgets for feature-following hierarchies}\label{sec:feature}

Codes differ in which cells of a feature-following hierarchy they advance. Patch-based codes that advance every cell of a level advance all of them. Codes that advance only the cells not covered by the next level advance a fraction $a=1-\rho/r^d=1-r^{\delta-d}$ of each coarser level, because the next level covers $\rho/r^d$ of its cells. Through $a$ the codimension $d-\delta$ of the feature enters the count of advanced cells and therefore the budget.

\begin{proposition}[Feature-following hierarchy]\label{prop:geom}
Suppose level $\ell$ holds $M_\ell=M_L\,\rho^{\ell-L}$ cells for $\ell=0,\dots,L$, with $\rho=r^\delta\ge1$, and the code advances all cells of the finest level and a fraction $a\in[0,1]$ of the cells of each coarser level. With uniform costs, $q_\ell=r^\ell$ and $H=0$, the ideal speedup is
\begin{equation}\label{eq:SL}
S_L=\frac{1+a\sum_{j=1}^{L}\rho^{-j}}{1+a\sum_{j=1}^{L}(r\rho)^{-j}},
\end{equation}
which for $a>0$ increases with $L$ and satisfies $1\le S_L\le r^L$, and which equals one when $a=0$. If $\delta>0$, then
\begin{equation}\label{eq:Sinf}
S_\infty=\lim_{L\to\infty}S_L=\frac{1+a/(\rho-1)}{1+a/(r\rho-1)},\qquad
\beta_\infty=1-\frac{1}{S_\infty}=\frac{a\,(r-1)\,\rho}{(\rho-1+a)(r\rho-1)},
\end{equation}
which for $a=1$ reduce to $S_\infty=1+(r-1)/[r(r^\delta-1)]$ and $\beta_\infty=(r-1)/(r^{\delta+1}-1)$. If $\delta=0$, then
\begin{equation}\label{eq:S0}
S_L=\frac{1+aL}{1+a\,(1-r^{-L})/(r-1)},
\end{equation}
which grows without bound, with increments that approach $a(r-1)/(r-1+a)$ per added level, or $1-1/r$ when $a=1$.
\end{proposition}

\begin{proof}
$S_L$ is a weighted mean of the values $1,r,\dots,r^L$, and the bounds, the monotonicity and the limits follow from that. Divide $\WG$ and $\WL$ by $QM_L=r^LM_L$ and set $j=L-\ell$ to obtain~\eqref{eq:SL}. The numerator equals $1+\sum_j a(r\rho)^{-j}r^j$, so $S_L$ is the mean of $1,r,\dots,r^L$ with positive weights $1,a(r\rho)^{-1},\dots,a(r\rho)^{-L}$, which gives the bounds; with $a=0$ only the first weight remains. Adding a level adds a term of value $r^{L+1}>S_L$ to this mean, so $S_L$ increases with $L$. For $\rho>1$ the two sums converge to $1/(\rho-1)$ and $1/(r\rho-1)$, which gives~\eqref{eq:Sinf} after simplification. For $\rho=1$ the numerator equals $1+aL$ and the denominator is a finite geometric sum, which gives~\eqref{eq:S0}, and the increments tend to $a(r-1)/(r-1+a)$.
\end{proof}

Real counts are geometric only on average, so the limit needs a statement under the hypotheses of Lemma~\ref{lem:growth}.

\begin{corollary}[Near-geometric counts]\label{cor:near}
Let $S_L^{\mathrm{geo}}$ be the value~\eqref{eq:SL} for exactly geometric counts, with the same $a$, $r$ and $\rho$. (i) If the held counts satisfy $k_1\rho^\ell\le M_\ell\le k_2\rho^\ell$ on every level $\ell=0,\dots,L$, then $S_L-1$ lies within a factor $\kappa=k_2/k_1$ of $S_L^{\mathrm{geo}}-1$, that is, $(S_L^{\mathrm{geo}}-1)/\kappa\le S_L-1\le\kappa\,(S_L^{\mathrm{geo}}-1)$. (ii) If $\delta>0$ and the level ratios $M_{\ell+1}/M_\ell$ converge to $\rho$ as $\ell$ grows, then $S_L$ converges to $S_\infty$ as the hierarchy deepens.
\end{corollary}

\begin{proof}
(i) The excess speedup is $S_L-1=\sum_\ell(Q-q_\ell)a_\ell M_\ell/\sum_\ell q_\ell a_\ell M_\ell$, where $a_\ell=a$ below the finest level and $a_L=1$. Both sums are positive combinations of the $M_\ell$, so replacing each $M_\ell$ by its bounds changes the ratio by at most the factor $\kappa$. (ii) After both sums are divided by $QM_L$, each term $M_{L-j}/M_L$ tends to $\rho^{-j}$. Since the ratios eventually exceed some $\theta>1$, the terms are bounded by a summable multiple of $\theta^{-j}$, and dominated convergence gives the limit.
\end{proof}

The factor $\kappa$ can be large when a full base grid is counted, as for the point of Section~\ref{sec:counts}, so the band in (i) bounds the order of magnitude rather than giving a sharp estimate.

Table~\ref{tab:limits} lists the limits for $r=2$ and $r=4$ under both conventions, Figure~\ref{fig:dimension} shows $\beta_\infty$ as a function of $\delta$, and Figure~\ref{fig:budget} shows how quickly the limits are reached when $a=1$. For surfaces and filled regions the budget lies within $1\%$ of its limit after four refined levels when $r=2$ and after two when $r=4$. For curves at $r=2$ the remaining gap halves with each level, so the approach is slower.

\begin{remark}[A full base grid]\label{rem:base}
Proposition~\ref{prop:geom} lets level 0 follow the geometric pattern, but a full base grid under a feature-following hierarchy often holds more cells than the pattern implies. Each extra cell advanced on level 0 adds $r^L$ to $\WG$ and only one to $\WL$, so it raises $S_L$. When $\delta>0$ this surplus fades like $\rho^{-L}$ relative to the work on the finest level. With $a=1$ and geometric refined levels, $S_L$ approaches its limit from above for large $L$ exactly when $M_0>M_1/(\rho-1)$, and from below otherwise. The reason is that in $\WG$ the base-grid surplus over an infinite geometric tail, $M_0-M_1/(\rho-1)$, decays like $\rho^{-L}$ relative to the finest level, while the corresponding correction in $\WL$ decays faster, like $(r\rho)^{-L}$. In either case the limits~\eqref{eq:Sinf} are values for deep hierarchies, not bounds.
\end{remark}

\begin{table}[ht]
\centering
\caption{Ideal speedup $S$ and overhead budget $\beta$ for a hierarchy that follows a feature of box-counting dimension $\delta$ (Proposition~\ref{prop:geom}), when every cell of every level is advanced ($a=1$), and the budget $\beta_{\mathrm{u}}$ when only uncovered cells are advanced in three dimensions ($a=1-r^{\delta-3}$). For $\delta>0$ the entries are the deep-hierarchy limits~\eqref{eq:Sinf}; for $\delta=0$ they are finite-depth values from~\eqref{eq:S0}, since the speedup grows without bound. For a filled region $a=0$ in the limit, so $\beta_{\mathrm{u}}=0$; at finite depth the uncovered shell around the region keeps $\beta_{\mathrm{u}}$ small but positive (Table~\ref{tab:counts}). The budget falls by roughly a factor $r$ for each added dimension of the refined feature, and advancing only uncovered cells lowers it further.}
\label{tab:limits}
\begin{tabular}{llcccccc}
\toprule
 & & \multicolumn{3}{c}{$r=2$} & \multicolumn{3}{c}{$r=4$}\\
\cmidrule(lr){3-5}\cmidrule(lr){6-8}
$\delta$ & Refined feature & $S$ & $\beta$ & $\beta_{\mathrm{u}}$ & $S$ & $\beta$ & $\beta_{\mathrm{u}}$\\
\midrule
0 & point, $L=3$ & 2.133 & 0.531 & 0.513 & 3.012 & 0.668 & 0.665\\
0 & point, $L=7$ & 4.016 & 0.751 & 0.738 & 6.000 & 0.833 & 0.832\\
1 & curve & $3/2$ & $1/3$ & $2/7$ & $5/4$ & $1/5$ & $4/21$\\
2 & surface & $7/6$ & $1/7$ & $4/49$ & $21/20$ & $1/21$ & $16/441$\\
3 & filled region & $15/14$ & $1/15$ & $0$ & $85/84$ & $1/85$ & $0$\\
\bottomrule
\end{tabular}
\end{table}

\begin{figure}[tbp]
\centering
\includegraphics[width=\linewidth]{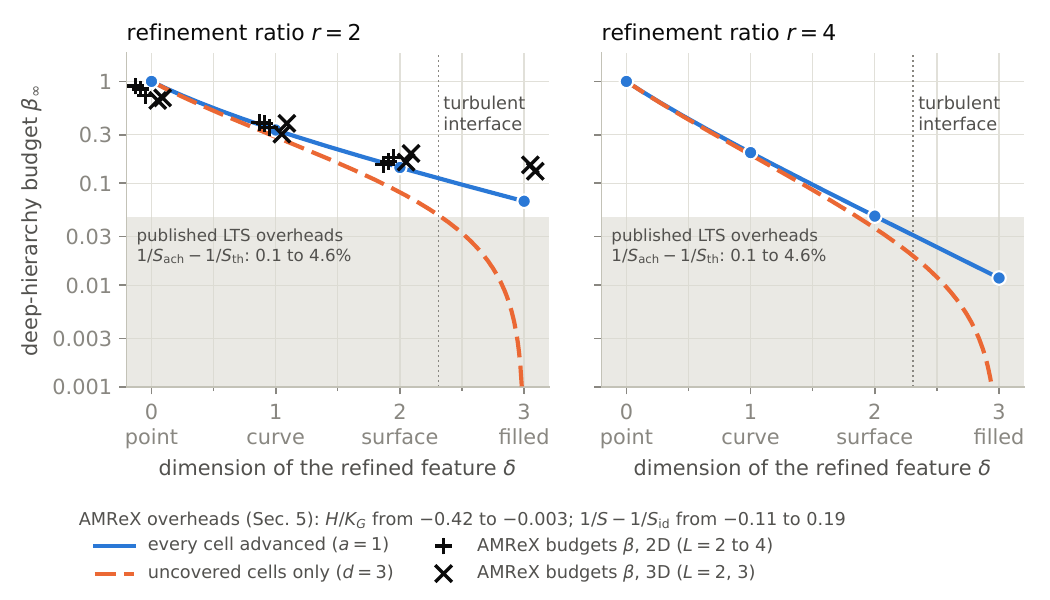}
\caption{Deep-hierarchy overhead budget $\beta_\infty$ against the box-counting dimension $\delta$ of the refined feature, for codes that advance every cell (solid) and codes that advance only uncovered cells in three dimensions (dashed), from~\eqref{eq:Sinf}. The gray band spans the overhead fractions $1/S_{\mathrm{ach}}-1/S_{\mathrm{th}}$ implied by published comparisons of achieved with theoretical local time stepping speedups (Section~\ref{sec:published}), 0.1 to 4.6 percent of the work; the dotted line marks $\delta\approx2.31$ for a turbulent interface~\cite{zhuang2018fractal}. Subcycling pays where the measured overhead lies below the curve: comfortably for points and curves, with some margin for surfaces at $r=2$, and marginally or not at all for surfaces at $r=4$ and for filled regions. Crosses in the left panel are the finite-depth budgets $\beta$ of the 17 AMReX hierarchies of Section~\ref{sec:amrex}, plotted on the axis of $\beta_\infty$ for comparison (upright for two dimensions, diagonal for three, offset sideways by depth). The axis names features in three dimensions, and the two-dimensional disk, a filled region of the plane, sits at $\delta=2$ beside the spheres. The overheads of the 17 hierarchies as Section~\ref{sec:two} defines them, from $-0.42$ to $-0.003$, lie below the axis. Computed from speedups like the published values, they range from $-0.11$ to $0.19$ (Section~\ref{sec:overhead}).}
\label{fig:dimension}
\end{figure}

\begin{figure}[tbp]
\centering
\includegraphics[width=\linewidth]{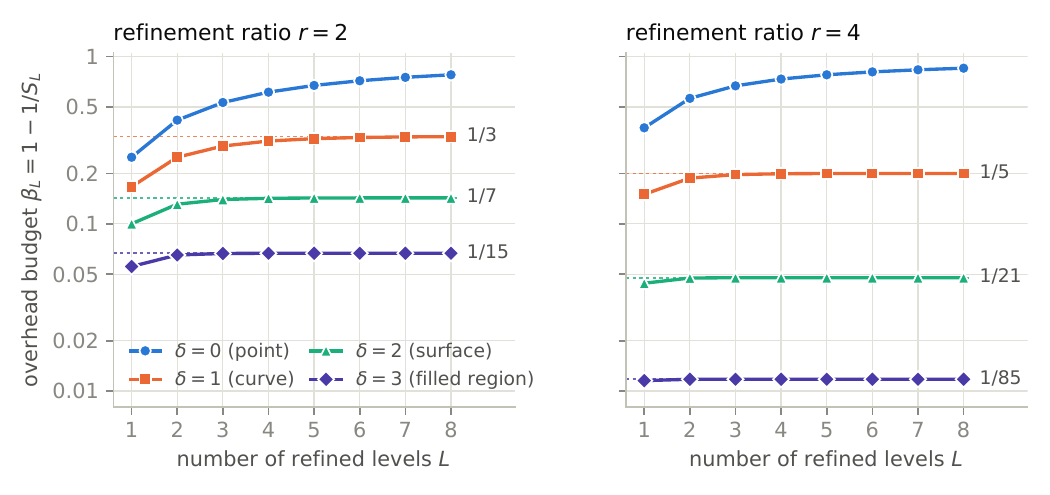}
\caption{Overhead budget $\beta_L=1-1/S_L$ from Proposition~\ref{prop:geom}, with every cell of every level advanced ($a=1$), against the number of refined levels for refined features of dimension $\delta=0,1,2,3$. Subcycling does less work than global stepping only if its overhead, as a fraction of the work of global stepping, lies below the curve. Dashed lines mark the limits $\beta_\infty=(r-1)/(r^{\delta+1}-1)$, and the vertical axis is logarithmic. For a point the budget approaches one as levels are added. For surfaces and filled regions it settles within a few levels, and for curves it approaches its limit more slowly.}
\label{fig:budget}
\end{figure}

These formulas move the dimension that matters from the space to what is refined. Table~\ref{tab:fixed} tempts one to conclude that subcycling suits one-dimensional problems and loses its value in three, but the calculation says otherwise. A point source or a compact object in three dimensions has $\delta=0$, and its savings grow with every level as they do in one dimension: with $r=2$, $a=1$ and seven refined levels, the ideal speedup is $4.02$.

A shock or a flame front in three dimensions is a surface with $\delta=2$, and in a deep hierarchy its budget approaches $1/7$ at $r=2$ however many levels are added. A larger refinement ratio lowers the plateau: with $a=1$ the excess speedup $S_\infty-1=(r-1)/[r(r^\delta-1)]$ decays like $r^{-\delta}$, because the finest level holds a larger share of the cells as $r$ increases.

\begin{table}[ht]
\centering
\caption{Codes that advance only uncovered cells, $r=2$: the per-level gain in ideal speedup for a point, $a(r-1)/(r-1+a)$, and the deep-hierarchy budget for a curve, both with $a=1-r^{\delta-d}$. A larger codimension $d-\delta$ helps.}
\label{tab:codim}
\begin{tabular}{lccc}
\toprule
 & $d=1$ & $d=2$ & $d=3$\\
\midrule
Point, gain per level & $1/3$ & $3/7$ & $7/15$\\
Curve, budget $\beta_{\mathrm{u}}$ & n/a & $2/9$ & $2/7$\\
\bottomrule
\end{tabular}
\end{table}

The dimension of the space enters in two ways: it bounds $\delta\le d$ and decides which features need refinement, since shocks, flames and material interfaces in three dimensions are surfaces. In codes that advance only uncovered cells it also enters through the codimension $d-\delta$ in $a$, and there a larger codimension helps. As Table~\ref{tab:codim} shows, a point gains more per level, and a curve keeps a larger budget, in three dimensions than in two.

\subsection{Agreement with earlier estimates}\label{sec:earlier}

The results of Section~\ref{sec:feature} agree with the earlier estimates: for deep hierarchies the case $a=1$ reproduces the estimates of Dursi and Zingale for a point and for a planar interface on patch meshes. The case $a=1-r^{\delta-d}$ reproduces their deep-hierarchy values for tree meshes, which include the per-level gains of Table~\ref{tab:codim} for a point and the budgets $2/9$ and $4/49$ for a planar interface in two and three dimensions. Their patch-mesh formula already implies the budgets $1/7$ and $1/21$ for a planar front in three dimensions. What is new here is the general dependence on $\delta$, $a$ and $r$, which covers curves in three dimensions, curved surfaces and fractal fronts. The Athena++ estimate assumes equal numbers of cells on every level and is the case $\delta=0$, $a=1$, $r=2$ of~\eqref{eq:S0}. The finding of Li and Hyman that the locked step cost only slightly more can be neither confirmed nor contradicted by the model, because the configurations of those tests are not recorded~\cite{li2003adaptive}. Hierarchies built around compact objects, such as the nested boxes used in numerical relativity~\cite{schnetter2004evolutions}, are close to the case $\delta=0$, where subcycling has the most to offer. The Athena++ authors likewise note that their single global time step becomes less efficient when more than ten levels cover one percent or less of the volume~\cite{stone2020athena}.

\subsection{Explicit counts}\label{sec:counts}

To test the growth law~\eqref{eq:growth} we built hierarchies on the unit cube with a $16^3$ base grid, $r=2$ and five refined levels. At each level we tagged the cells whose centers lie at distance at most two cell widths from the feature and replaced each tagged cell by its eight children on the next level. The features were a point at $c=(0.5123,\,0.4871,\,0.5037)$, a circle of radius $0.3$ about $c$ in a plane through $c$ tilted by $0.37$ radians from the grid plane $z=0.5037$, and a sphere and a solid ball of radius $R=0.3$ about $c$. The distance is the Euclidean distance to the point, the circle and the sphere, and to the ball it is the distance outside it, zero inside. The small offsets of $c$ from the center of the cube keep the features off the grid nodes. Table~\ref{tab:counts} lists the counts, which grow by factors close to 2 and 4 per level for the circle and the sphere. The counts for the point stay nearly constant, with fluctuations of up to about a quarter from level to level as cells enter and leave the band. For the ball the ratio is close to $8[(R+2\Delta x_\ell)/(R+4\Delta x_\ell)]^3$, which approaches 8 slowly and reaches $7.42$ on the last level, because the band is not thin compared with $R$ on the coarse levels.

\begin{table}[ht]
\centering
\small
\setlength{\tabcolsep}{4pt}
\caption{Cells per level, $M_\ell$, in synthetic three-dimensional hierarchies with a $16^3$ base grid ($M_0=4096$), $r=2$ and a tagging band of two cell widths. The mean growth factor $(M_5/M_1)^{1/4}$ is close to $r^\delta$ except for the ball, whose band is not thin on the coarse levels. The speedup $S_5$ evaluates~\eqref{eq:multiwork} on the counts with $H=0$ and every cell advanced; the model column is~\eqref{eq:SL} at the same depth; $\beta_{\mathrm{u}}$ is the budget on the same counts when only uncovered cells are advanced, to be compared with the limits $2/7$, $4/49$ and $0$ of Table~\ref{tab:limits}.}
\label{tab:counts}
\begin{tabular}{lcrrrrrcccc}
\toprule
Feature & $\delta$ & $M_1$ & $M_2$ & $M_3$ & $M_4$ & $M_5$ & $(M_5/M_1)^{1/4}$ & $S_5$, counts & $S_5$, model & $\beta_{\mathrm{u}}$\\
\midrule
point  & 0 & 288 & 272 & 304 & 240 & 304 & 1.01 & 8.094 & 3.048 & 0.879\\
circle & 1 & 2\,984 & 6\,072 & 12\,144 & 24\,416 & 48\,344 & 2.01 & 1.516 & 1.477 & 0.298\\
sphere & 2 & 9\,808 & 37\,616 & 148\,728 & 592\,808 & 2\,372\,680 & 3.94 & 1.167 & 1.166 & 0.082\\
ball   & 3 & 10\,552 & 52\,376 & 319\,272 & 2\,209\,632 & 16\,396\,936 & 6.28 & 1.080 & 1.071 & 0.008\\
\bottomrule
\end{tabular}
\end{table}

Evaluating the identity~\eqref{eq:multiwork} on these counts gives $S_5=1.516$, $1.167$ and $1.080$ for the circle, the sphere and the ball, slightly above the limits $3/2$, $7/6$ and $15/14$, an excess that the full base grid explains (Remark~\ref{rem:base}). For the sphere the speedups computed from the counts rise to $1.183$ at $L=2$ and then fall to $1.167$ at $L=5$, and for the ball they fall from $1.163$ at $L=1$ to $1.080$. For the circle they rise to $1.539$ at $L=3$ and then fall toward $3/2$, reaching $1.500$ at $L=10$. The uncovered-cell budgets on the same counts, $0.298$ and $0.082$ for the circle and the sphere, are close to the limits $2/7$ and $4/49$, and the ball's $0.008$ is the small positive value left by its uncovered shell. For the point the base grid dominates, since its $4096$ cells outnumber the $240$ to $304$ cells on each refined level, and the counts give $S_5=8.09$ against $3.05$ for the geometric model. A band of three cell widths gives mean growth factors $1.00$, $2.00$ and $3.88$ for the point, circle and sphere. The circle refined with $r=4$ grows by $4.07$ and $3.98$ per level, so the exponent does not depend on these choices.

\begin{remark}[A circle aligned with the grid]\label{rem:aligned}
The same circle in the grid plane $z=0.5037$ shows why Corollary~\ref{cor:near}(i) rather than a limit is the right statement for aligned features. Its level ratios swing between $1.79$ and $2.25$ out to eleven levels, because the plane's position within a cell changes from level to level, and $S_L$ oscillates between $1.491$ and $1.525$ for $5\le L\le 11$ without settling. Its refined levels stay within a factor $1.13$ of geometric growth against $1.02$ for the tilted circle, which converges to $3/2$.
\end{remark}

For a particular hierarchy the identity evaluated on its actual counts is the right tool, while Proposition~\ref{prop:geom} gives the limit as the hierarchy deepens. The blocked hierarchies that AMReX builds at depths two to four grow more slowly than these synthetic counts (Section~\ref{sec:counts-amrex}).

\section{Published overheads and what the work model leaves out}\label{sec:overheads}

A budget decides the comparison of work only when it is set against the overhead $H$, as a fraction of the work of global stepping. This section sets published estimates of $H$ against the budgets of Section~\ref{sec:dimension} and then states what the work model leaves out.

\subsection{Published overheads against the budgets}\label{sec:published}

Comparisons of achieved with theoretical speedups give a first estimate of $H$: if $S_{\mathrm{th}}=\WG/(\WL-H)$ is the speedup predicted by the update counts and $S_{\mathrm{ach}}=\WG/\WL$ the measured one, then $H/\WG=1/S_{\mathrm{ach}}-1/S_{\mathrm{th}}$. Breuer and Heinecke report $2.14$ against a theoretical $2.28$ and $6.0$ against $6.3$ for local time stepping in a discontinuous Galerkin seismology code, which gives $H/\WG\approx0.029$ and $0.008$~\cite{breuer2022next}. Seny et al.\ report $10.457$ against a theoretical $10.932$ for multirate discontinuous Galerkin, or $0.004$~\cite{seny2014efficient}. Fernando and Sundar's four octree runs give values from $0.001$ to $0.046$~\cite{fernando2022scalable}. These codes are not block-structured AMR and the measured time includes load imbalance, but they suggest overheads of a few percent of the work. Figure~\ref{fig:dimension} shows them as a band. Against Table~\ref{tab:limits}, such overheads leave a front in three dimensions profitable at $r=2$, with budgets $1/7\approx0.14$ when every cell is advanced and $4/49\approx0.08$ when only uncovered cells are, and close to break-even at $r=4$, with $1/21\approx0.048$ and, for the illustrative case of uncovered cells only, $16/441\approx0.036$. Measurements in AMR codes also report large gains in some settings: adaptive time stepping for the implicit radiation solver of RAMSES reduced the computing time by factors of 5 to 50 depending on the problem~\cite{commercon2014fast}, and the GAMER-2 authors expect notable gains when fine levels cover small subvolumes~\cite{schive2018gamer}. For the incompressible solver IAMR, built on AMReX, subcycling was slower for a fixed number of steps and faster for a fixed final time~\cite{liu2024investigate}. Section~\ref{sec:amrex} measures $H$ directly in one AMReX code, where it is negative as Section~\ref{sec:two} defines it and mostly positive when computed from speedups in this way.

\subsection{What the work model leaves out}\label{sec:limits}

\paragraph{The overhead.}
The model places every difference between the two schedules other than the update count in $H$, and several of its parts have a definite sign. Subcycling must interpolate coarse data in time to fill the ghost cells of each fine substep~\cite{mccorquodale2011high}, and it must keep coarse data at more than one time level~\cite{dursi2005efficiency}. For $\delta>0$ the extra storage is about the coarse share of the cells, $1/\rho$ of the state in a deep hierarchy. Against these, global stepping fills the coarse ghost cells $r$ times per coarse interval rather than once, a cost that subcycling avoids. Conservation requires flux corrections at coarse--fine boundaries under either schedule~\cite{berger1989local}, so only the difference in that bookkeeping enters $H$. The largest contributions arise when the equations contain elliptic constraints or implicit terms. Subcycled projection methods for incompressible flow~\cite{almgren1998conservative,martin2008cell} and for low Mach number reacting flow~\cite{day2000numerical} advance the levels separately and then correct them with synchronization solves across levels, whereas a global step can solve one composite problem. For such methods $H$ may be a large fraction of the work, and the small budgets of Table~\ref{tab:limits} leave little room for it.

\paragraph{Time on parallel machines.}
The model counts work, not time, and both schedules take $Q$ sequential steps on the finest level for each coarse step. When the cost of a step is set by communication latency rather than by arithmetic, as in strong scaling to many processors, the wall-clock time of either schedule is bounded below by that of the $Q$ finest-level steps. Subcycling then saves little time even when it saves work, and the Athena++ authors note that its savings may be offset by reduced parallel efficiency~\cite{stone2020athena}. How the comparison comes out in the latency-bound limit depends on how a code organizes its levels. Subcycling advances each level in a phase of its own, $\sum_\ell q_\ell$ phases per coarse step. If global stepping advances all levels in one fused phase per step, it has $Q$ phases, and subcycling can be slower by a factor of up to $\sum_\ell q_\ell/Q\to r/(r-1)$, which is 2 at $r=2$. If global stepping instead advances each level in its own phase, it has $(L+1)Q$ phases, and subcycling has fewer: $63$ against $192$ at $r=2$ with five refined levels. On accelerators a level advance also pays a launch cost, and a small level does not fill the device, so advancing level $\ell$ once costs about $\alpha+C\max(N_\ell,N_{\mathrm{fill}})$. The small coarse levels then cost more per cell, which shrinks the avoided work and the budget. SeisSol's experimental automatic choice of its local time stepping schedule minimizes a cost per unit simulated time of exactly this form~\cite{seissol2026local}. Subcycling also complicates load balance~\cite{dubey2014survey,stone2020athena,schive2018gamer}: finer levels do more of the work, and different levels are active at different moments. In this regime the comparison needs a timing model with a latency charge for each synchronization phase, not a work count alone.

\paragraph{Accuracy and stability.}
The comparison is made at equal final time rather than at equal error, and global stepping advances level $\ell$ at a Courant number smaller than subcycling does by the factor $Q/q_\ell$. For the first-order upwind method applied to $u_t+cu_x=0$, the modified equation carries the numerical diffusion coefficient $c\,\Delta x\,(1-\nu)/2$, where $\nu$ is the Courant number \cite[Sec.~8.6]{leveque2002finite}. For this method the smaller steps of global stepping make the coarse solution more diffusive, not more accurate. For higher-order methods whose temporal error matters, smaller steps can reduce the error instead. Subcycling adds its own error through the interpolation in time of coarse ghost data~\cite{stone2020athena}. With a multistage Runge--Kutta method every stage needs ghost data, and Li and Hyman noted that interpolation between coarse solutions stored only at $t_n$ and $t_{n+1}$ is at most second order. They locked the time step so that the order at coarse--fine boundaries depended only on the spatial interpolation~\cite{li2003adaptive}. McCorquodale and Colella later reached fourth order by interpolating with the Runge--Kutta stage values~\cite{mccorquodale2011high}. A complete comparison would set the two schedules at matched error, which the work model does not attempt.

\section{A paired test in AMReX}\label{sec:amrex}

\subsection{Design}\label{sec:design}

A direct test of the accounting runs the same hierarchy with subcycling on and off in one code. For each level it records the cells advanced, the steps taken and the time spent advancing, filling ghost cells, correcting fluxes and synchronizing. The per-level timings of the update kernels give the costs $C_\ell$, and the difference between the two runs in the time spent beyond those costs gives $H$. The model passes if costs and an overhead fitted on some runs predict the measured speedup on others. A single subcycled run gives a weaker test that needs no second run. Let $K$ be its time in the update kernels and $O$ its time outside them. If every update costs the same on every level and under both schedules, global stepping spends $\Sid K$ in its kernels alone, and subcycling wins whenever $O<(\Sid-1)K$, where $\Sid$ comes from the level counts the code already knows. The condition is sufficient but not necessary, because it charges subcycling with all of its time outside the kernels, including the part that global stepping pays as well. The kernels must be large enough that the work per cell dominates their time rather than the fixed cost of each launch or loop, or the measured $C_\ell$ will not describe production runs.

We ran the paired test in the tutorial code \texttt{Advection\_AmrCore} of AMReX~\cite{zhang2019amrex}, which advances a passive scalar with an explicit finite-volume method. The code offers both schedules and advances every cell of every level, so $a=1$ in the notation of Section~\ref{sec:feature}. A patch applied the tagging rule of Section~\ref{sec:growth} with a band of two cell widths around a point, a circle, a sphere or a filled region placed as in Section~\ref{sec:counts}. It replaced the rotating velocity by a constant one so that $Q=r^L$ holds exactly. It also added timers for the update kernel, ghost filling, flux synchronization, averaging down, the stability check and the time-step computation. The hierarchy was built once at initialization with $r=2$, a blocking factor of 8, one buffer cell and base grids of $128^2$ and $32^3$ cells. Every run used one process and one thread on a shared machine with two cores.

An analysis plan fixed the configurations, timers, models, selection rule and criterion before any timed run. Appendix~\ref{app:protocol} records its hash, its one amendment and the checks on the data. The primary set has 17 configurations. In two dimensions a point, a circle and a filled disk are refined to depths $L=2$, 3 and 4, and in three dimensions a point, a tilted circle, a sphere and a ball to $L=2$ and 3. Each configuration ran five times under each schedule in a fixed pseudorandom order, and we report medians.

The measured speedup is the ratio of window times, $S=T_G/T_L$, with subscripts $G$ and $L$ for global stepping and subcycling (this $L$, for local time stepping, is not the depth). Let $U_G$ and $U_L$ be the numbers of cell updates and $C=K_G/U_G$ the cost per update of global stepping. The overhead of Section~\ref{sec:two} is then $H=T_L-T_G+C(U_G-U_L)$, reported as $x_C=H/K_G$ in units of kernel time. Since $T_L<T_G$ precisely when $x_C<\beta$, comparing $x_C$ with $\beta$ is the same test as $S>1$, an identity rather than a prediction. Because $T_G$ also contains time outside the kernels, $x_C$ does not enter the speedup formula~\eqref{eq:speed}; the time saving obeys $1-1/S=\beta-y$ instead, where $y=1/S-1/\Sid$. The plan defined the overhead as $O_L-O_G$ alone. The three measures differ in what they count as overhead.
\begin{center}
\small
\begin{tabular}{@{}>{\raggedright\arraybackslash}p{0.17\linewidth}>{\raggedright\arraybackslash}p{0.47\linewidth}>{\raggedright\arraybackslash}p{0.28\linewidth}@{}}
\toprule
Measure & What it counts & Where it is used\\
\midrule
$x_C=H/K_G$ & the difference in time outside the kernels, plus the kernel time of subcycling beyond its updates priced at $C$; in units of $K_G$ & the overhead of Section~\ref{sec:two}, compared with $\beta$\\
$(O_L-O_G)/K_G$ & the difference in time outside the kernels alone, in units of $K_G$ & the plan's measure, reported in Section~\ref{sec:overhead}\\
$y=1/S-1/\Sid$ & all time beyond the update counts, $H+\beta O_G$, which includes a share $\beta$ of the time global stepping spends outside the kernels; as a fraction of $T_G$ & the measure of the published values (Section~\ref{sec:published})\\
\bottomrule
\end{tabular}
\end{center}

The split between fitting and testing is by geometry. The ten configurations tagged around points and filled regions fit the model, and the seven tagged around curves and surfaces, the cases the budgets of Table~\ref{tab:limits} concern, are held out. The held-out ideal speedups of 1.19 to 1.66 and growth factors of 1.12 to 2.86 lie inside the training ranges of 1.15 to 10.3 and 0.44 to 3.86. On each measure six of the seven fall between the training clusters of points and filled regions, where no training configuration lies, and five of the seven do so on both. Twelve candidate models pair one of four kernel models, from one cost per update to one per update on each level, with one of three overhead models, from a fixed fraction of the kernel time to a cost per level advance plus costs per ghost cell and per cell. Coefficients were fitted by non-negative least squares on relative residuals, and the pair with the lowest median leave-one-configuration-out error on the training set was selected. It combines one cost per update in each dimension with an overhead for each schedule and dimension built from a cost per level advance, per ghost cell and per cell. Its predictions for the held-out configurations were frozen, and their hash recorded, before any held-out speedup was computed. The criterion was a median relative error $e=|\hat S-S|/S$ of at most 10 percent on the held-out set, with every error above 20 percent explained.

\subsection{Predicted and measured speedups}\label{sec:prediction}

The selected model passed. The median held-out error was 0.071 with a 95 percent bootstrap interval of 0.021 to 0.082, and the largest error was 0.080 for the sphere at $L=2$ (Figure~\ref{fig:validation} and Table~\ref{tab:amrex}). The interval reflects variation between repetitions only, for these seven configurations on one machine, and the estimate lies near its upper end. Four cells were rerun under the plan, and with their original repetitions instead the median held-out error is 0.075 and the largest 0.110. No held-out error exceeded 20 percent, so the explanation standard of the plan was never invoked.

The remaining comparisons were made after the data, and they show that the test does not separate the fitted model from the update counts. The ideal speedup $\Sid$ of the counts alone predicts the held-out speedups with a median error of 0.073, and all twelve candidate pairs pass with held-out medians from 0.032 to 0.071, the selected pair having the largest. Measured on the saving $S-1$ rather than on $S$, the counts and the selected model have median held-out errors of 21 and 22 percent, so a tolerance of 10 percent on $S$ is lenient when $S$ lies between 1.2 and 1.7. The fit also gets the speedup without getting its parts. The frozen model underpredicted every held-out window time by 5 to 45 percent under both schedules. In all four three-dimensional held-out configurations it predicted more time outside the kernels under subcycling than under global stepping, the opposite of what was measured, and these errors largely cancel in the ratio. The model predicts the ratio of the two times about as well as the update counts, which need no fitting, and the test does not show that the fitted costs or the fitted overhead transfer.

The counts fail for the points, where they overstate the speedup by 20 to 97 percent. The measured $S$ is a weighted mean of the kernel ratio $K_G/K_L$ and the non-kernel ratio $O_G/O_L$ with weights $K_L$ and $O_L$, and the kernel ratio lies between 0.80 and 1.00 times $\Sid$. For the points the non-kernel ratio is far below $\Sid$, so $S$ falls to between 0.51 and 0.83 of $\Sid$ and the fitted overhead is needed: left out one at a time the points were predicted with errors from 0.005 to 0.39. For the filled regions the non-kernel ratio is larger, so $S$ exceeds $\Sid$. The break-even condition~\eqref{eq:break} concerns the difference of the two times and is not affected. As Proposition~\ref{prop:geom} predicts, the measured speedups grow with depth for the points (from 2.17 to 5.23 in two dimensions and from 2.09 to 2.67 in three), while for the other features they stay between 1.18 and 1.59. At these depths the data do not resolve the ordering among features with $\delta>0$.

\begin{table}[ht]
\centering
\footnotesize
\setlength{\tabcolsep}{3.2pt}
\caption{Paired AMReX runs with $r=2$, every cell advanced and a fixed hierarchy; medians of five repetitions. $S=T_G/T_L$ is the measured speedup and $e=|\hat S-S|/S$ the error of the selected model: the frozen prediction for the held-out curves and surfaces, and the leave-one-configuration-out prediction, marked $^{*}$, for the training points and filled regions. $e_{\mathrm{id}}$ is the error when $\Sid$ itself is the prediction, with no fitted coefficient. $\Sid$ and $\beta$ come from the level counts with uniform cost, $x_C=H/K_G$ is the overhead as Section~\ref{sec:two} defines it, with one cost per update taken from global stepping, $y=1/S-1/\Sid$ is the overhead measured from speedups as for the published codes, and the one-run test is the sufficient condition of Section~\ref{sec:design}. Rows are generated by \texttt{validation/s6\_tables.py}. The growth factor is $(M_L/M_1)^{1/(L-1)}$, to be compared with $r^\delta=2^\delta$. Subcycling was faster in every configuration.}
\label{tab:amrex}
\begin{tabular}{llccccccccccc}
\toprule
$d$ & Feature & $\delta$ & $L$ & $S$ & $e$ & $e_{\mathrm{id}}$ & $\Sid$ & $\beta$ & $x_C$ & $y$ & One-run test & Growth\\
\midrule
2 & point & 0 & 2 & 2.172 & 0.390$^{*}$ & 0.708 & 3.710 & 0.730 & $-0.060$ & $0.191$ & yes & 0.44\\
2 & point & 0 & 3 & 3.678 & 0.163$^{*}$ & 0.741 & 6.405 & 0.844 & $-0.275$ & $0.116$ & yes & 0.82\\
2 & point & 0 & 4 & 5.232 & 0.081$^{*}$ & 0.968 & 10.296 & 0.903 & $-0.422$ & $0.094$ & yes & 0.76\\
2 & circle & 1 & 2 & 1.367 & 0.071 & 0.128 & 1.541 & 0.351 & $-0.229$ & $0.083$ & no & 1.61\\
2 & circle & 1 & 3 & 1.594 & 0.014 & 0.024 & 1.631 & 0.387 & $-0.420$ & $0.014$ & no & 1.58\\
2 & circle & 1 & 4 & 1.504 & 0.071 & 0.103 & 1.659 & 0.397 & $-0.392$ & $0.062$ & no & 1.68\\
2 & disk & 2 & 2 & 1.291 & 0.132$^{*}$ & 0.056 & 1.219 & 0.179 & $-0.192$ & $-0.046$ & no & 3.55\\
2 & disk & 2 & 3 & 1.387 & 0.159$^{*}$ & 0.137 & 1.197 & 0.165 & $-0.362$ & $-0.114$ & no & 3.59\\
2 & disk & 2 & 4 & 1.316 & 0.115$^{*}$ & 0.103 & 1.180 & 0.153 & $-0.298$ & $-0.087$ & no & 3.69\\
\addlinespace
3 & point & 0 & 2 & 2.092 & 0.258$^{*}$ & 0.366 & 2.857 & 0.650 & $-0.003$ & $0.128$ & yes & 1.00\\
3 & point & 0 & 3 & 2.669 & 0.005$^{*}$ & 0.203 & 3.211 & 0.689 & $-0.159$ & $0.063$ & yes & 0.88\\
3 & circle & 1 & 2 & 1.280 & 0.001 & 0.122 & 1.437 & 0.304 & $-0.016$ & $0.085$ & no & 1.12\\
3 & circle & 1 & 3 & 1.521 & 0.075 & 0.073 & 1.632 & 0.387 & $-0.132$ & $0.045$ & yes & 1.32\\
3 & sphere & 2 & 2 & 1.191 & 0.080 & 0.001 & 1.192 & 0.161 & $-0.068$ & $0.001$ & no & 2.69\\
3 & sphere & 2 & 3 & 1.222 & 0.058 & 0.018 & 1.244 & 0.196 & $-0.065$ & $0.015$ & no & 2.86\\
3 & ball & 3 & 2 & 1.213 & 0.030$^{*}$ & 0.029 & 1.179 & 0.151 & $-0.100$ & $-0.024$ & no & 2.81\\
3 & ball & 3 & 3 & 1.182 & 0.114$^{*}$ & 0.026 & 1.151 & 0.131 & $-0.078$ & $-0.023$ & no & 3.86\\
\bottomrule
\end{tabular}
\end{table}

\begin{figure}[tbp]
\centering
\includegraphics[width=\linewidth]{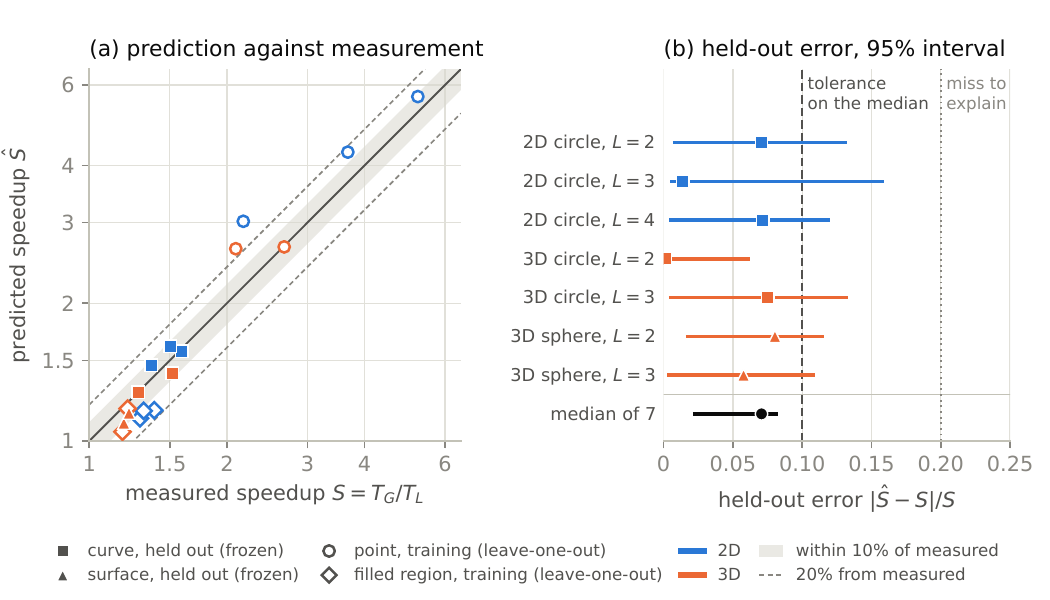}
\caption{Validation in AMReX. (a) Predicted against measured speedup on logarithmic axes. Filled markers are the frozen predictions for the held-out curves (squares) and surfaces (triangles), and open markers are the leave-one-configuration-out predictions for the training points (circles) and filled regions (diamonds). Color marks the dimension. The shaded band lies within 10 percent of the measured value, and the dashed lines lie 20 percent from it. (b) Held-out errors with 95 percent bootstrap intervals and their median, against the tolerance of 10 percent on the median and the level of 20 percent above which a miss must be explained.}
\label{fig:validation}
\end{figure}

\subsection{The overhead}\label{sec:overhead}

A negative overhead is allowed by the model, and the measures in this subsection were again defined after the data. As Section~\ref{sec:two} defines it, the overhead of subcycling was negative in every configuration, with $x_C$ from $-0.42$ to $-0.003$. It has two parts of opposite sign. The first is the time outside the kernels, which is smaller under subcycling: $(O_L-O_G)/K_G$ runs from $-0.45$ to $-0.06$. This is the measure the plan fixed, so the sign does not depend on the change of definition. Pooled over the 17 configurations, the time-step computation, which global stepping repeats at every step of the finest level and subcycling performs once per coarse step, accounts for 68 percent of $O_G-O_L$. Ghost filling adds 31 percent and averaging down 18 percent because global stepping does both on every level at every fine step, and flux synchronization and untimed time add 11 percent. The stability check that only the subcycled advance performs in this code offsets 29 percent. The second part is the kernel itself, which costs 0 to 25 percent more per update under subcycling. The subcycled kernel fuses the flux and update loops and copies the fluxes for the flux registers. That copying is flux bookkeeping that Section~\ref{sec:two} counts in $H$ rather than in $C$, a measured case of the schedule-dependent costs of Remark~\ref{rem:schedule}.

The sign rests on how this tutorial organizes its time step. The time-step computation takes 7 to 13 percent of $T_G$. Without its difference between the schedules $x_C$ would be positive in 6 of the 17 configurations, and without the stability check as well in 3 (Table~\ref{tab:counterfactual}). Even then the counterfactual overheads $x_C'$ and $x_C''$ of Table~\ref{tab:counterfactual} stay below $\beta$ in all 17, by at least 0.08 in units of $K_G$, so on this post hoc accounting a global step that computed its time step once per coarse step would still have been slower in every configuration. The margins are wide partly because global stepping in this code also does ghost filling and averaging down on every level at each step of the finest level, so the overhead favored subcycling despite the costlier kernel. Reaching break-even would take a larger cost that falls on subcycling, such as the load imbalance between levels on many processors (Section~\ref{sec:limits}), which a run on one core does not have.

The sign also depends on how the overhead is measured. The published values of Section~\ref{sec:published} come from speedups as $1/S_{\mathrm{ach}}-1/S_{\mathrm{th}}$, which counts as overhead all time beyond the update counts, including time outside the kernels that both schedules spend alike. The same formula on the AMReX runs, $1/S-1/\Sid$, gives $-0.11$ to $0.19$. It is positive in 12 of the 17 configurations (every point, curve and sphere) and above the published maximum of 0.046 in 8 of them. A comparison with the published band of Figure~\ref{fig:dimension} must use this measure, and on it this code overlaps and exceeds the published range.

\subsection{The one-run test}\label{sec:onerun}

The single-run test of Section~\ref{sec:design} held in 6 of the 17 configurations, the five points and the three-dimensional circle at $L=3$, and subcycling won in each. Global stepping never won, so the data cannot reveal a false positive. In the other 11 configurations subcycling also won, but the test charges subcycling with all of its non-kernel time $O_L$, here 0.29 to 1.74 times its kernel time, while the overhead that decides the comparison is negative. Weighting $\Sid$ by the per-level kernel costs of the same run gave the same verdicts. The subcycled kernel cost up to 25 percent more per update, so the test could have erred toward subcycling; it did not, because the time outside the kernels was large under both schedules. In this code the test settled the points and was too conservative for most other features. Where it is silent, a short paired run with subcycling off decides directly through $S$.

\subsection{Level counts in AMReX}\label{sec:counts-amrex}

Lemma~\ref{lem:growth} concerns deep hierarchies, and the AMReX hierarchies are shallow and blocked. Their mean growth factors lie at or below $r^\delta$ for every feature and below it for every feature with $\delta>0$. In three dimensions those with $\delta>0$ also lie below the growth factors of the synthetic hierarchies of Table~\ref{tab:counts}, which tag without blocking, and the difference is padding: the blocking factor, the buffer cell and the clustering of tagged cells into boxes add cells around the feature on each level, and at these depths that padding is wide compared with the band itself (Appendix~\ref{app:protocol}). Section~\ref{sec:growth} absorbs such padding into the constants of the lemma, so it does not change the limit, but it gives the coarse refined levels more cells than the scaling implies and raises the budgets. Figure~\ref{fig:dimension} shows the AMReX budgets above the deep-hierarchy curve for the surfaces and filled regions and close to it for the curves. That closeness reflects offsetting effects rather than the limit, and where level 1 fills the domain the label $\delta$ is nominal. Depths of two to four levels are too shallow to measure the exponent itself.

\subsection{Scope}\label{sec:scope}

Before any timing the analysis plan limited the test to one code, one explicit kernel, one core and a fixed hierarchy in the primary set. It does not test time on parallel machines, elliptic synchronization, accelerator launch costs or accuracy at matched error, which Section~\ref{sec:limits} places outside the work model. On a shared machine with two cores the absolute times say nothing about production hardware. No configuration came near break-even: the realized saving $1-1/S$ was at least 0.15 and the margin $\beta-x_C$ above 0.20. The test therefore concerns the prediction of the speedup, not the decision at break-even. In the test itself, the selected pair won by a median leave-one-out error of 0.124 against 0.125, and reruns in a later session moved two held-out speedups by 4 and 11 percent (Appendix~\ref{app:protocol}). The runs show that the overhead as Section~\ref{sec:two} defines it can favor subcycling in a real code, but not that it will do so in a parallel production run. Six exploratory configurations run after the predictions were frozen (Appendix~\ref{app:protocol}) play no part in the verdict, and subcycling was faster in all six.

\section{Conclusions}\label{sec:conclusions}

Subcycling pays exactly when its extra cost of coordinating the levels, as a fraction of the work of global stepping, stays below $\beta=1-1/\Sid$. The geometry of the refined set, together with which cells the code advances, largely sets $\beta$: when every cell is advanced and the counts grow geometrically, point-like refinement can repay a substantial overhead, while refinement around a surface in three dimensions can save about one seventh of the work in a deep hierarchy at $r=2$ and one twenty-first at $r=4$. In one AMReX code on one core a cost model fixed in advance predicted the speedups for curves and surfaces with a median error of 7 percent, a post hoc check found that the update counts alone did as well except for points, and no configuration came near break-even.

These statements concern work. Wall-clock time, accuracy and the cost of elliptic synchronization can shift the balance in either direction. AMReX and Nyx already choose the subcycling pattern at run time by minimizing the estimated update work per unit of simulated time~\cite{almgren2013nyx,amrex2026source}. SeisSol does the same for its clusters, with optional launch and fill costs~\cite{seissol2026local}. Adding a measured overhead to such a search turns the budget into a decision a code can make for itself.

Of the three questions that remain open, the first is the schedule $q$ that minimizes $\WL(q)+H(q)$ when $H$ and a latency charge per phase are measured. Its work-only version is already solved by enumeration in AMReX, so the open part is the overhead. The second is the behavior of $S_L$ for features aligned with the grid and for lattice self-similar fractals, where the counts oscillate and only the band of Corollary~\ref{cor:near} holds. The third is the comparison of the two schedules at matched error rather than at equal final time. A test near break-even, with refinement ratio 4 around surfaces and filled regions and a global step that computes its time step once per coarse step, or on many processors, would show whether a measured overhead decides the comparison where it matters.

What a developer gains over timing both schedules once is a sequence that starts before any run. The budgets of Table~\ref{tab:limits}, summarized in the box of Section~\ref{sec:intro}, say, from the dimension of the refined feature, the refinement ratio and which cells the code advances, whether subcycling is worth trying. When an update costs the same under both schedules, one profiled subcycled run can show that subcycling pays; in the AMReX runs it did so for every point and was silent for most other features. When it is silent, a short paired run with subcycling off decides directly, and the margin $\beta-x_C$ then says how much more overhead the code can absorb before the verdict reverses. The first step rests on one fact: in a deep hierarchy the update work that subcycling avoids is set mainly by the dimension of what is refined, not of the space.

\section*{Reproducibility}

The script \texttt{hyman2026local\_reproduce.py} (Python~3.13 with NumPy~2.5 and Matplotlib~3.11) regenerates Tables~\ref{tab:fixed}--\ref{tab:counts} and Figures~\ref{fig:dimension} and~\ref{fig:budget}. It checks Propositions~\ref{prop:break} and~\ref{prop:geom}, the identity~\eqref{eq:multibreak}, Corollary~\ref{cor:near}(i), Remark~\ref{rem:norm}, the single-run test and the capped time ratio in exact rational arithmetic over a grid of parameters. It also builds the synthetic hierarchies of Section~\ref{sec:counts}, including the aligned circle of Remark~\ref{rem:aligned}, and computes the overhead fractions of Section~\ref{sec:published} from the published speedup pairs, using no random numbers and running in seconds. The AMReX points of Figure~\ref{fig:dimension} are copied from the evaluation of Section~\ref{sec:amrex}, and the script checks them against it. The folder \texttt{validation} holds that evaluation: the analysis plan and its amendment, the deviation log, the patch with the pinned AMReX commits and build notes, the run script, the raw timings, the fitting and evaluation scripts (which also need SciPy~1.18), the frozen predictions, the generator of Figure~\ref{fig:validation} and the scripts \texttt{s6\_tables.py} and \texttt{s8\_appendix\_table.py}, which produce Tables~\ref{tab:amrex} and~\ref{tab:counterfactual}. The repository is at \url{https://github.com/machyman/hyman2026local}.

\section*{Use of AI tools}

The research question and the dimensional hypothesis are the author's, and an earlier draft developed the two-level and multilevel work models with the assistance of ChatGPT (OpenAI). Claude (Anthropic) checked the derivations, located and verified the prior work, extended the analysis to feature-following hierarchies, wrote the counting and plotting scripts and revised the text. For Section~\ref{sec:amrex}, Claude also patched and built the AMReX harness, wrote the analysis plan, ran the paired timings, fitted and evaluated the cost models and drafted the section, which three independent Claude referee passes and a citation check then reviewed. An independent Claude referee panel proposed the bounded-ratio corollary, the single-run test, the normalization remark and the two-sided latency count, which Claude then verified and wrote into the text. The author checked the derivations and computations and takes full responsibility for the content of this paper.

\appendix
\renewcommand{\thetable}{A\arabic{table}}
\setcounter{table}{0}
\renewcommand{\theHtable}{A\arabic{table}}

\section{AMReX test protocol}\label{app:protocol}

The analysis plan was written and its SHA-256 hash (beginning \texttt{e4c615e3}) recorded before any timed run. One amendment (hash beginning \texttt{cbc2f22c}) was written before any primary timing. It set the length of the timing window for each configuration so that every subcycled window lasts about two seconds or more. It followed one subcycled calibration run of each of the 17 configurations, held-out ones included, with no held-out speedup computed. It also followed smoke runs on training configurations, in which the kernel time per update differed between the schedules by about 15 percent on some levels. Neither changed the models or the split. The plan, its amendment, the deviation log and the data are in the \texttt{validation} folder.

All 190 runs completed (170 in the primary set and 20 reruns) with timers that accounted for 95.9 to 99.9 percent of each window. The two schedules built identical hierarchies, and the conserved sum drifted by at most $1.3\times10^{-13}$ relative to its value. Four cells whose repetitions varied by 5 percent or more (population standard deviation) were rerun once as the plan required, and the reruns replace the originals. The reruns ran in a later session and moved $S$ for the two-dimensional circle by 4 percent at $L=2$ and 11 percent at $L=3$. The global schedule of the two-dimensional circle at $L=3$ still varied by 7.3 percent after its rerun and is kept with a flag.

Non-negative least squares set 5 of the selected pair's 12 overhead coefficients to zero, so the fitted overhead of global stepping is a cost per cell in both dimensions. In three dimensions each schedule's three overhead coefficients rest on four training configurations, and on three in each leave-one-out fold. With 14 coefficients the pair had a median leave-one-out error of 0.124 on the training set, against 0.125 for the next pair. Its predictions for the held-out configurations were written to a file whose hash (beginning \texttt{11d2e030}) was recorded before any held-out speedup was computed. The bootstrap interval of Section~\ref{sec:prediction} is a percentile interval of the held-out median from 2,000 resamples of the repetitions, each followed by a refit of the selected pair. An independent recomputation from the raw timings reproduced the hash, the coefficients and the errors to a relative difference of $2\times10^{-14}$.

After the predictions were frozen we ran six exploratory configurations, which play no part in the verdict. They are the two-dimensional point and circle with $r=4$ at $L=2$, the circle at $L=3$ with regridding on in two and three dimensions, and deeper points at $L=6$ in two dimensions and $L=4$ in three. With its training coefficients unchanged, the selected model predicted their speedups with errors from 0.001 to 0.136, the largest for the point with $r=4$. With regridding on, the predictions used the counts of the initial hierarchy rather than those of later ones. Subcycling again spent less time outside its kernels, and it was faster in all six. Four of these cells varied by 5 percent or more and were not rerun, and the timers of the deep two-dimensional point accounted for as little as 94.8 percent of the window.

The level counts of the AMReX hierarchies are in the \texttt{validation} folder. Their mean growth factors are 1.58 to 1.68 for the circle and 3.55 to 3.69 for the disk in two dimensions, and 1.12 to 1.32 for the circle, 2.69 to 2.86 for the sphere and 2.81 to 3.86 for the ball in three. The synthetic hierarchies of Table~\ref{tab:counts} tag cell by cell without blocking and grow by 2.00, 3.95 and 6.10 for the circle, the sphere and the ball between $64^3$ and $128^3$ cells across the domain, against 1.60, 3.19 and 3.38 in AMReX at $L=3$. Level 1 covers 24 to 47 percent of the domain for the circle and the disk in two dimensions and 38 to 100 percent for the circle, the sphere and the ball in three, where it fills the domain for the sphere and the ball at $L=3$.

Table~\ref{tab:counterfactual} gives two counterfactual overheads for each configuration, computed after the data from the same timers. Each adds to $x_C$ the difference between the schedules' median times in the omitted timers, divided by $K_G$, and leaves every other time unchanged. The rows are generated by \texttt{validation/s8\_appendix\_table.py}.

\begin{table}[ht]
\centering
\small
\caption{Counterfactual overheads for the 17 AMReX configurations, post hoc. $x_C$ is the overhead of Table~\ref{tab:amrex}; $x_C'$ removes the difference between the schedules in the time-step computation, and $x_C''$ removes the difference in the stability check as well. Subcycling would win where the value lies below $\beta$, which holds in every row.}
\label{tab:counterfactual}
\begin{tabular}{llcccccc}
\toprule
$d$ & Feature & $\delta$ & $L$ & $\beta$ & $x_C$ & $x_C'$ & $x_C''$\\
\midrule
2 & point & 0 & 2 & 0.730 & $-0.060$ & $0.075$ & $0.035$\\
2 & point & 0 & 3 & 0.844 & $-0.275$ & $-0.115$ & $-0.138$\\
2 & point & 0 & 4 & 0.903 & $-0.422$ & $-0.250$ & $-0.265$\\
2 & circle & 1 & 2 & 0.351 & $-0.229$ & $-0.088$ & $-0.179$\\
2 & circle & 1 & 3 & 0.387 & $-0.420$ & $-0.226$ & $-0.305$\\
2 & circle & 1 & 4 & 0.397 & $-0.392$ & $-0.169$ & $-0.257$\\
2 & disk & 2 & 2 & 0.179 & $-0.192$ & $-0.024$ & $-0.139$\\
2 & disk & 2 & 3 & 0.165 & $-0.362$ & $-0.157$ & $-0.278$\\
2 & disk & 2 & 4 & 0.153 & $-0.298$ & $-0.070$ & $-0.194$\\
\addlinespace
3 & point & 0 & 2 & 0.650 & $-0.003$ & $0.063$ & $0.041$\\
3 & point & 0 & 3 & 0.689 & $-0.159$ & $-0.067$ & $-0.086$\\
3 & circle & 1 & 2 & 0.304 & $-0.016$ & $0.083$ & $0.028$\\
3 & circle & 1 & 3 & 0.387 & $-0.132$ & $-0.015$ & $-0.065$\\
3 & sphere & 2 & 2 & 0.161 & $-0.068$ & $0.031$ & $-0.034$\\
3 & sphere & 2 & 3 & 0.196 & $-0.065$ & $0.051$ & $-0.012$\\
3 & ball & 3 & 2 & 0.151 & $-0.100$ & $-0.007$ & $-0.072$\\
3 & ball & 3 & 3 & 0.131 & $-0.078$ & $0.045$ & $-0.025$\\
\bottomrule
\end{tabular}
\end{table}

\clearpage
\bibliographystyle{unsrtnat}
\bibliography{hyman2026local}

\end{document}